\documentclass[a4paper]{article}
\usepackage{amsmath}
\usepackage{amsfonts}
\usepackage{amssymb}
\usepackage{amsthm}
\usepackage{textcomp}
\usepackage{tipa}
\usepackage{fullpage}
\usepackage{graphicx}
\usepackage{color}
\usepackage{url}

\newcommand{\comments}[1]{}

\renewcommand{\bar}{\overline}
\renewcommand{\geq}{\geqslant}

\newcommand{\w}{\omega}

\newcommand{\al}{\alpha}
\newcommand{\be}{\beta}

\newcommand{\ka}{\kappa}

\newcommand{\R}{\mathbb{R}}

\newcommand{\Q}{\mathbb{Q}}
\newcommand{\I}{\mathbb{I}}
\newcommand{\N}{\mathbb{N}}
\newcommand{\Sorg}{\mathbb{S}}

\newcommand{\B}{\mathcal{B}}
\newcommand{\U}{\mathcal{U}}

\newcommand{\V}{\mathcal{V}}

\newcommand{\Lin}{Lindel\"of }
\newcommand{\Lind}{Lindel\"of}

\newcommand{\n}{\frac{1}{n}}

\swapnumbers
\newtheorem{thm} {Theorem}[section]
\newtheorem{propn}[thm]{Proposition}
\newtheorem{lem} [thm]{Lemma}

\theoremstyle{definition}
\newtheorem{defn}[thm]{Definition}
\newtheorem{xmpl}[thm]{Example}

\newtheorem*{construction}{Construction}

\newtheorem*{oq}{Open Question}
\newtheorem*{ackn}{Acknowledgements}

\theoremstyle{remark}

\newtheorem{nt}[thm]{Note}

\title{A Note on Quasi-\Lin spaces }
\author{Petra Staynova\\ Oxford University, Pembroke College}
\date{}

\begin{document}
\maketitle

\begin{abstract}
The quasi-\Lin property was first introduced by Arhangelski in \cite{Arc}, as a strengthening of the weakly \Lin property. 
However, unlike \Lin and weakly \Lin spaces, very little is known about how quasi-\Lin spaces behave under the main topological operations, and how the property relates to separation axioms. 
In the present paper, we look at several properties of quasi-\Lin spaces. 
We  consider several examples: a weakly \Lin space which is not quasi-\Lind, a product of \Lin spaces which is not even  quasi-\Lind, and a quasi-\Lin space which is not ccc. 
At the end, we pose some open questions.
\end{abstract}

\begin{nt}
All spaces are assumed Hausdorff.
\end{nt}

\section{Introduction}
It is well known that any product of compact spaces is compact, and that the product of even two \Lin spaces need not be \Lin \cite{ENG}.
Various generalisations of \Lin spaces have been considered throughout years and attempts to compare their behaviour under basic topological operations with the behaviour of \Lin spaces have been made.
One such generalisation - the notion of a weakly \Lin space - was introduced by Z. Frolik \cite{Fro}. 

\begin{defn}
A topological space $X$ is called \emph{weakly \Lind} if for any open cover $\U$ of $X$, one can find a countable subfamily  $\U'\subseteq\U$ such that $X=\bar{\bigcup \U'}$.
\end{defn}

Unfortunately, unlike compactness and \Lind ness, that property is not inherited by closed subspaces. 
 Thus A. Arhangelski \cite{Arc} considered a stronger property:
\begin{defn}
A topological space $X$ is called \emph{quasi-\Lin} if for any closed subset $C\subseteq X$ and any family $\U$ of open in $X$ sets which cover $C$, a countable subfamily $\U'\subseteq\U$ can be found such that $C\subset\bar{\bigcup\U'}$.
\end{defn}

In fact, Arhangelski defined a more general invariant - the \emph{quasi-\Lind} number $\textrm{qL}(X)$ of a given topological space $X$:
\begin{defn}
$\textrm{qL}(X)=\w .\min\{\tau: \forall\textrm { closed } C\subset X , \forall \U \subset\tau_X \textrm{ with } C\subseteq \bigcup\U, \exists \textrm{ a countable } 
\linebreak \U'\subseteq\U \textrm{ such that } C\subset\bar{\bigcup\U'}\}$.
\end{defn}

He used this to generalise a theorem of Bell, Ginsburgh and Woods \cite{BGW} for obtaining an upper bound of the cardinality of a topological space.

It follows directly from the definition that any \Lin space is quasi-\Lin and any quasi-\Lin space is weakly \Lind. 
The uncountable discrete subspace has neither of the above properties.
We use the following proposition to give a non-trivial example of a quasi-\Lin space which is not \Lind.

\begin{propn}
Every separable topological space $X$ is quasi-\Lind.
\end{propn}
\begin{proof}
Indeed, let $C\subset X$ be closed and let $\U$ be a family of open subsets with $C\subset\bigcup\U$. 
Take a countable dense subset $A \subset X$ and let $A_1=A\cap(\bigcup\U)=\{a_1,\ldots,a_n,\ldots\}$.
For every $n$ choose $U_n\in\U$ such that $a_n\in U_n$.
 Consider $V=\bigcup\U \setminus\bar {A_1}$ and note that $V$ is open and $V \cap A = \emptyset$. 
Hence, $V$ must be the empty set. 
Therefore,  $\bigcup\U \subset \bar{A_1}$.
Furthermore, $A_1 \subset \bigcup_{n\in\N}U_n \subset \bigcup\U$. 
So, we get  $\bar{\bigcup\U}=\bar{A_1}$
 and, moreover,
    $C \subset \bar{\bigcup\U} = \bar{\bigcup_{n\in\N}U_n}$. 
Therefore, $C$ is weakly \Lind, as required. 
\end{proof}

Using this proposition, we can give the following two examples of quasi-\Lin spaces which are not \Lind:

\begin{xmpl}
The Sorgenfrey plane $\Sorg\times\Sorg$ is not \Lind, but is separable and hence quasi-\Lind.
\end{xmpl}

\begin{xmpl}
The Niemytzki plane $L$ is separable, hence quasi-\Lind, but not \Lind.
\end{xmpl}

The quasi-\Lin and weakly \Lin properties coincide in the case of normal spaces (see \cite{PS} for the proof). 
We use ideas from Mysior (\cite{M}) and modify a construction from \cite{SZ} in order to obtain the following example:

\begin{xmpl} \label{0ex1}  There exists a 
weakly Lindel\"of space $X$ which is not quasi-Lindel\"of (and not even \Lind).
\end{xmpl}

\begin{construction}
Let $A=\{(a_\al, -1):\al<\w_1\}$ be an $\w_1$-long sequence in the set $\{(x,-1):x\geq 0\}\subseteq\R^2$.
Let $Y=\{(a_\al, n):\al<\w_1,n\in\w\}$. 
Let $a=(-1,-1)$.
Finally let $X=Y\cup A\cup\{a\}$.

We topologize $X$ as follows:
\begin{itemize}
\item[-]
all points in $Y$ are isolated;
\item[-]
for $\al<\w_1$ the basic neighborhoods of $(a_\al,-1)$ will be of the form 
\begin{displaymath}
U_n(a_\al,-1)=\{(a_\al,-1)\}\cup\{(a_\al, m):m\geq n\} \textrm{ for } n\in\w
\end{displaymath}

\item[-]
the basic neighborhoods of $a=(-1,-1)$ are of the form
\begin{displaymath}
U_\al (a)=\{a\}\cup\{(a_\be , n):\be>\al, n\in\w\} \textrm{ for } \al<\w_1.
\end{displaymath}
\end{itemize}
%
%
%
%
Let us point out that $A$ is closed and discrete in this topology. 
Indeed, 
for any point $x\in X$ there is a basic neighborhood $U(x)$ such that $A\cap U(x)$ contains at most one point and also that $X\setminus A=\{a\}\cup Y$ is open (because $U_\al(a)\subset Y\cup\{a\}$.
Hence $X$ contains an uncontable closed discrete subset and therefore  it cannot be \Lind. 

Note that for any open $U\ni a$ the  set $X\setminus\bar{U}$ is at most countable. 
Indeed, for any $\al<\w_1$, $\bar{U_\al(a)}=U_\al(a)\cup\{(a_\be, -1):\be>\al\}$. 
Hence $X\setminus \bar{U_\al(a)}$ is at most countable.

It is easily seen that $X$ is Hausdorff. 
Without much effort, it can also be proven that $X$ is  Urysohn.

Let us now prove that $X$ is weakly \Lind.
Let $\U$ be an open cover of $X$. 
Then there exists a $U(a)\in \U$ such that $a\in U(a)$.
We can find a basic neighborhood $U_\be(a)\subset U(a)$.
Then  $\bar{U_\be(a)}\subset \bar{U(a)}$ and hence $X\setminus \bar{U(a)}$ will also be at most  countable.
Hence $X\setminus \bar{U(a)}$ can be covered by (at most) countably many elements of $\U$, say $\U^{*}$. Set $\U'=U^{*}\cup\{U(a)\}$.
Then, $X\subseteq\bigcup_{U\in\U'} \bar{U}\subseteq\bar{\bigcup_{U\in\U'}U}$.
Therefore, $X$ is weakly \Lind.

Now, let us prove that $X$ is not quasi-\Lind.
Consider the 1-neighborhood of $a$: $U_1 (a)=\{a\}\cup\{(a_\be , n):\be>1, n\in\w\} $. 
We have that $C=X\setminus U_1(a)$ is closed. 
We show the uncountable family of basic open sets $\U=\{U_0(a_\al,-1):\al<\w_1\}$ forms an open cover of $C$ which has no countable subcover with dense union. 
Note that the sets $U_0(a_\al,-1)$ are closed and open. 
Indeed, $X\setminus U_0(a_\al,-1)=\bigcup\{U_0(a_\be,-1):\be\neq\al\}\cup  U_{\al+1}(a)$. 
Hence, if we remove even one of the $U_0(a_\al,-1)$, the point $(a_\al,-1)$ would remain uncovered. 
Therefore, $X$ is not quasi-\Lind.

\comments{
$X$ is Urysohn because 
\begin{displaymath}
\bar{U_\al(a)}=U_\al(a)\cup\{(a_\be,-1):\be>\al\},
\end{displaymath}
\begin{displaymath}
\bar{U_n(a_\al,-1)}=U_n(a_\al,-1)\cup\{a\} \textrm{, and }
\end{displaymath}
\begin{displaymath}
\bar{U(a_\al,n)}=U(a_\al,n)=\{(a_\al, n)\}.
\end{displaymath}
}
 
\end{construction}

\Lind ness is equivalent to requiring that every cover of basic open sets has a countable subcover. 
\comments{We define a similar property for quasi-\Lin spaces:

\begin{defn}
A space $X$ is called \emph{base-quasi-\Lin}if for every closed $C\subseteq X$ and every basic open family $\U$ of subsets of $X$ there is a countable subfamily $\U'\subset\U$ such that $C\subset\bar{\bigcup\U'}$.
\end{defn}

This definition does not depend on the choice of the base $\B$ of $X$. 
Indeed, let $X$ be $\B$-quasi-\Lin and $\B'$ be any other base. 
Let $C\subseteq X$ be closed and let $C \subset\bigcup\U$, $\U\subset\B'$. 
For every $U\in\U$ there exists $\V_U\subset\B$ such that $U=\bigcup\V_U$. 
Hence $C\subset\bigcup_{U\in\U}(\bigcup\V_U)$ is a family of $\B$-open sets covering $C$. 
Since $X$ is $\B$-quasi-\Lin there is a countable subfamily 
\begin{displaymath}
\V'=\{V_n:n\in\N\}
\subseteq\bigcup_{U\in\U}\left(\bigcup\V_U\right)=
\bigcup\U
\end{displaymath}
such that $C\subset\bar{\bigcup\V'}$. 
For every $n\in\N$ let $U_n\in\U$ be such that $V_n\subset U_n$ and let $\U'=\{U_n:n\in\N\}$. 
Then $\U'$ is the required subfamily of $\U$ since $C\subset\bar{\bigcup\V'}\subset\bar{\bigcup\U'}$. 
Hence $X$ is $\B'$-quasi-\Lind.

In \Lin spaces, base-\Lin coincides with \Lind. 
} 
We have a similar result here:

\begin{propn}
Let $X$ be a topological space. The following are equivalent:
\begin{enumerate}
\item $X$ is quasi-\Lind.
\item Let $\B$ be a fixed base for $X$. Then for any closed subset $C\subset X$
and any cover $\U$ of $C$ with $\U \subset \B$ there is a countable subfamily $\U'$ 
of $\U$ such that $C\subset\bar{\bigcup\U'}$.
\end{enumerate}
\end{propn}
\begin{proof}
The forward direction is trivial.

For the converse, let $C\subset X$ be closed and let $\U$ be a family of open subsets of $X$ covering $C$, i.e. $C\subset\bigcup\U$. 
Let $\B$ be any base for the topology of $X$. 
For every $U\in\U$ there is a family $\V_U\subset\B$ such that $U=\bigcup\V_U$. 
Then $C\subset\bigcup\U=\bigcup_{U\in\U}(\bigcup\V_U)$. 
Since $X$ is base quasi-\Lind, then there exists a countable 
\begin{displaymath}
\V'=\{V_n:n\in\N\}\subset
\bigcup_{U\in\U}\left(\bigcup\V_n\right)=
\bigcup\U
\end{displaymath}
such that $C\subset\bar{\bigcup\V'}$. 
As above, choose a countable $\U'\subseteq\U$ such that $\bigcup\V'\subset\bigcup\U'$.
Then $C\subset\bar{\bigcup\U'}$ and hence $X$ is quasi-\Lind.
\end{proof}

Note that this is independent of choice of basis, since if $\B_1$ and $\B_2$ are two bases then  the respective (2)-conditions are both equivalent to (1), and hence also equivalent to each other.

The following Theorem is proved in  \cite{PS}:

\begin{thm}
If $X$ satisfies the countable chain condition (i.e. is ccc), then $X$ is quasi-\Lind.
\end{thm}

This shows that the ccc property implies the quasi-\Lin property, which in turn implies that the space is weakly-\Lind.

The converse, however, does not hold, as the following example shows.

\begin{xmpl}[\cite{SS}]
The lexicographic square is quasi-\Lin (in fact, it is compact), but not ccc.
\end{xmpl}

As we pointed out in the beginning, products of compact spaces is compact, and a product of two \Lin spaces might not be \Lind. 
Such products might not even be weakly \Lind, as the following example from \cite{HJ} shows:

\begin{xmpl}
There is a topological space $X$ that is not weakly \Lin (and hence not quasi-\Lind) that is a product of two \Lin spaces.
\end{xmpl}

Hence neither the weakly \Lin nor the quasi-\Lin property is productive, i.e. both spaces have the same behaviour with respect to products as the \Lin property. 
For weakly \Lin spaces, we  have the following result:

\begin{propn}
If $X$ is weakly \Lin and $Y$ is compact, then $X\times Y$ is weakly \Lind.
\end{propn}

A proof of this can be found in \cite{PS}. 
It is natural to ask whether this can be extended to quasi-\Lin spaces, namely: 
\begin{oq}
Is the product of a quasi-\Lin space $X$ and a compact space $Y$ quasi-\Lind? 
\end{oq}
This question is interesting even in the partical case:
\begin{oq}
Is the product of the unit interval $[0,1]$ with a quasi-\Lin space, quasi-\Lind?
\end{oq}

The following proposition is  a very special case of the first question:

\begin{propn}
If $A=\{0, 1, 2,\ldots, n\}$ is a finite discrete set and $Y$ is a quasi-\Lin space, then the product $A\times Y$ is quasi-\Lind.
\end{propn}

This can be proved by induction; the key step it to prove this for a two-point discrete set: 

\begin{lem}\label{refprop}
If $X$ is quasi-\Lin and $Y=\{0,1\}$, then $X\times Y$ is quasi-\Lind.
\end{lem}

\begin{proof}
Let $C \subset X \times Y$ be closed and $\U$ be an open cover of $C$. Consider $C_0=\{x \in X \colon (x, 0) \in C\}$
and  $C_1=\{x \in X \colon (x, 1) \in C\}$. Then both $C_0$ and $C_1$ are closed in $X$ since $C$ is closed. Moreover,
$C \subset (C_0\times\{0\})  \cup (C_1 \times \{1\})$.  Set $\U_0=\{U \in \U \colon U \cap (C_0 \times \{0\}) \not = \emptyset\}$.
 Clearly, $\U_0$ is an open cover of $C_0 \times \{0\}$.
 Since  $C_0$ is closed and $X$ is quasi-\Lind  \ we find a countable subfamily, say $\U'$, of $\U_0$ such
 that $C_0 \times \{0\} \subset \bar{\bigcup \U'}$ (here the identification of  $C_0$ and $C_0 \times\{0\}$ is obvious).
 Likewise,  we  deal with $C_1$ and find a countable subfamily,  say $\U^*$, of $\U$ such
 that $C_1 \subset \bar{\bigcup \U^*}$. Then  $\hat{\U}=\U' \cup \U^*$ is as required, i.e.
  $C \subset \bar{\bigcup \hat{\U}}$.
\end{proof}

\begin{ackn} 
The author is indebted to the referee for their thorough reading of the paper, and for their very helpful remarks. 

In the initially submitted version of the paper, the unit interval [0,1] was mistyped as \{0,1\} in the second open question. It was thus pointed out by the referee that that Lemma \ref{refprop} holds.  Here, we provide their proof, which they generously offered. 
\end{ackn}

\comments{
The following example provides a negative answer to the question. 
Is is based on an example given in \cite{BGW} showing that there is a Hausdorff, first countable, weakly \Lin space with arbitrarily large cardinality. 

\begin{xmpl}
There is a topological space $X$ which is a product of a weakly \Lin and a compact space and is not quasi-\Lind.
\end{xmpl}

\begin{proof}
Let $Z$ be the space constructed in \cite{BGW}.  
Then $X=Z\times[0,1]$ with the usual product topology. 
We shall show that $X$ is not quasi-\Lind. 
Let us first recall the construction of $Z$:
We will denote the irrational numbers by $\I$. 
Let $\kappa$ be an arbitrary uncountable cardinal number, $A\subset\I$ be countable and dense in $\I$ (and hence also dense in $\R$). 
Let $Z=(\Q\times\ka)\cup A$.
Note $A\cap \Q=\emptyset$.
For every $(q,\al)\in\Q\times\ka$ define a neighborhood base
\begin{displaymath}
U_n(q,\al)=\{(r,\al):r\in\Q\wedge|r-q|<\n   \},\ n\in\N.
\end{displaymath}
For every $a\in A$, define a neighborhood base
\begin{displaymath}
U_n(a)=\{b\in A:|b-a|<\n\}\cup\{(q,\al):\al<\ka\wedge|q-a|<\n\}, \ n\in\N.
\end{displaymath}

Let $\tau$ be the topology generated by $\{U_n(q,\al):(q,\al)\in\Q\times\ka\}\cup \{ U_n(a):a\in A, n\in\N\}$.
In \cite{BGW} it was proved that $Z$ is a Hausdorff, weakly \Lin not \Lin space. 
We prove in addition that $Z$ is not quasi-\Lind. 

\end{proof}
}

\bibliographystyle{alpha}
\bibliography{Bibliography}

\begin{flushright}
Petra Staynova

Pembroke College

Oxford University

 Oxford, OX1 1DW

United Kingdom

e-mail: \url{petra.staynova@pmb.ox.ac.uk}
\end{flushright}
\end{document}